\documentclass[11pt]{amsart}

\usepackage{amscd,amsmath,amssymb,amsfonts,amsthm,mathrsfs, todonotes}
\usepackage[a4paper,text={160mm,240mm},centering,headsep=5mm,footskip=10mm]{geometry}

\usepackage[OT2,T1]{fontenc}
\DeclareSymbolFont{cyrletters}{OT2}{wncyr}{m}{n}
\DeclareMathSymbol{\Sha}{\mathalpha}{cyrletters}{"58}

\usepackage[colorlinks,linkcolor=black,anchorcolor=black,citecolor=black]{hyperref}

\usepackage{tikz-cd}
\usepackage[all,cmtip]{xy}

\usepackage{array}
\usepackage{color}

\numberwithin{equation}{section}
\newtheorem{theorem}{Theorem}[section]
\newtheorem{lemma}[theorem]{Lemma}

\newtheorem{proposition}[theorem]{Proposition}
\newtheorem{corollary}[theorem]{Corollary}

\newtheorem{definition}[theorem]{Definition}
\newtheorem{remark}[theorem]{Remark}

\newtheorem*{notat*}{Notation}

\newcommand{\Z}{\mathbb{Z}}

\renewcommand{\O}{\mathcal{O}}

\renewcommand{\r}{\rightarrow}

\newcommand{\et}{\mathrm{\acute{e}t}}
\newcommand{\Spec}{\mathrm{Spec}}

\newcommand{\ul}{\underline}

\DeclareMathOperator*{\colim}{colim}

\begin{document}
\title[On the Gersten conjecture]{On the relative Gersten conjecture for Milnor $K$-theory in the smooth case}
\author{Morten Lüders}
\address{Universität Heidelberg,
Mathematisches Institut,
Im Neuenheimer Feld 205,
69120 Heidelberg,
Germany}
\email{mlueders@mathi.uni-heidelberg.de}
%\thanks{The author is supported by a DFG research fellowship.}
\subjclass[2010]{19D45; Higher symbols, Milnor $K$-theory.}
\thanks{The author is supported by the DFG Research Fellowship LU 2418/1-1.}

\begin{abstract} 
%We prove a statement about the relative Gersten conjecture for Milnor K-theory, that is for the Milnor K-sheaf on smooth schemes over discrete valuation rings. 
We show that the Gersten complex for the (improved) Milnor K-sheaf on a smooth scheme over an excellent discrete valuation ring is exact except at the first place and that exactness at the first place may be checked at the discrete valuation ring associated to the the generic point of the special fibre. This complements results of Gillet-Levine for $K$-theory, Geisser for motivic cohomology and Schmidt-Strunk and the author for \'etale cohomology. \end{abstract}

\maketitle
\setcounter{page}{1}
\setcounter{section}{0}
\pagenumbering{arabic}

\section{Introduction} 
Let $\mathcal{O}_K$ be a discrete valuation ring with residue field $k$ and local parameter $\pi$. Let $X$ be a smooth scheme of dimension $d$ over $\Spec(\mathcal{O}_K)$. In \cite{GL87} Gillet-Levine show that the Gersten complex
$$0\rightarrow \mathcal{K}_{n,X}\xrightarrow{} \bigoplus_{x\in X^{(0)}}i_{x*} K_n(x)\rightarrow \bigoplus_{x\in X^{(1)}} i_{x*}K_{n-1}(x)\rightarrow...\r \bigoplus_{x\in X^{(d)}}i_{x*}K_{n-d}(x)\r 0 $$
for algebraic $K$-theory is exact except at the first and third place. The Gersten conjecture for algebraic $K$-theory says that this complex is exact everywhere and Gillet-Levine show that this conjecture holds if it holds for the discrete valuation ring associated to the generic point of the special fibre of $X$. There are analogous results for motivic cohomology by Geisser \cite[Sec. 4]{Ge04}, for Bloch-Ogus theory by the author \cite{LudersBlochOgus} and under some additional assumptions by Schmidt-Strunk for $\mathbb{A}^1$-invariant cohomology theories with Nisnevich descent such as \'etale cohomology with torsion coefficients prime to the residue characteristic \cite{SS19}. We will say that in these cases the relative Gersten conjecture holds.
 
In this article we prove the analogous result for Milnor $K$-theory. Let $\hat{K}^M_n$ denote the improved Milnor $K$-theory defined in \cite{Ke10} and $\hat{\mathcal{K}}^M_{n,X}$ its sheafification. $\hat{K}^M_n$ coincides with usual Milnor $K$-theory for fields and local rings with big residue fields. The main difference between (improved) Milnor $K$-theory and the other theories mentioned above is that it does not have a localization sequence. 
\begin{theorem}(Corollary \ref{maincorollary})
Let the notation be as above. Assume furthermore that $\O_K$ is excellent. Then the sequence of sheaves
$$\hat{\mathcal{K}}^M_{n,X}\xrightarrow{i} \bigoplus_{x\in X^{(0)}}i_{x*}K^M_n(x)\rightarrow \bigoplus_{x\in X^{(1)}}i_{x*}K^M_{n-1}(x)\r... \bigoplus_{x\in X^{(d)}}i_{x*}K^M_{n-d}(x)\r 0$$
is exact. Furthermore, if $\O_{X,\eta}$ is the discrete valuation ring induced by a generic point of the special fibre, then $i$ is injective if the map
$$\hat{K}^M_{n}(\O_{X,\eta})\r K^M_{n}(\mathrm{Frac}(\O_{X,\eta}))$$
is injective for all such $\eta$.
\end{theorem}
Philosophically the situation may be summed up as follows: the Gersten complex for Milnor $K$-theory is exact when restricting to elements which are, or can be moved into, horizontal position relative to the base. This reduces the exactness of the Gersten complex to the case of a discrete valuation ring in which case the Gersten complex is known to be exact except at the first place. This will get clearer in the proofs. 

%\begin{corollary}\label{theorem_KM3}
%Let the notation be as above and $m>1$. Then the Gersten conjecture holds for the sheaf $\hat{\mathcal{K}}^M_{n,X}/m$ if $n\leq 3$.
%\end{corollary}

A few remarks are in order:
\begin{remark}\label{remark_introduction}
\begin{enumerate}
\item If $\O_K$ is a field, then the Gersten conjecture holds for (improved) Milnor $K$-theory: Kerz \cite{Ke09,Ke10} proved exactness at the first place, Elbaz-Vincent-Müller-Stach \cite{ElbazMuller2002} and Gabber (unpublished) proved exactness at the second place and again Gabber proved exactness on the right (see \cite[Sec.  6]{Ro96}). If $\O_K$ is a field, then the Gersten conjecture holds furthermore for motivic cohomology, higher Chow groups, algebraic $K$-theory and \'etale cohomology with torsion coefficients by work of Voevodsky \cite{VSF00}, Bloch \cite{Bl86}, Quillen \cite{Qu72} and Bloch-Ogus \cite{BO74}, respectively. Using a method of Panin \cite{Pa03}, these results can be extended to any local ring $A$ which is regular, connected and equicharacteristic. 
%We should also mention the Gersten conjecture for higher Chow groups proved by Bloch \cite{Bl86}. 
For an axiomatic setup see the article \cite{CHK97} of Colliot-Th\'el\`ene-Hoobler-Kahn.
%\item The Gersten conjecture holds integrally for motivic cohomology in degrees $(n,n)$ for smooth schemes over discrete valuation rings by work of Geisser \cite[Thm. 4.1]{Ge04}.
\item For smooth schemes over discrete valuation rings the relative Gersten conjecture for algebraic $K$-theory and motivic cohomology implies the Gersten conjecture for these theories with finite coefficients (see Geisser \cite[Thm. 8.2]{GL00} and Geisser-Levine \cite[Cor. 4.3]{Ge04}). These applications use that Milnor $K$-theory has the special property that one can ``lift symbols'' from the third to the second place of the Gersten complex for a DVR.
One cannot prove an analogous result for Milnor $K$-theory with finite coefficients due to the mentioned lack of a localization sequence. However, by comparison with $K$-theory one may obtain the following weaker result (see \cite[Prop. 3.2]{Lu17'} by the author): let $(m,(n-1)!)=1$ (for example $m=p^r$ and $p>(n-1)$). Then the Gersten conjecture holds for $\hat{\mathcal{K}}^M_n/m$.
\item Let $\O_{X,x}^h$ be a henselian local ring with residue field $k(x)$ associated to a point $x$ of a smooth scheme $X$ over  a discrete valuation ring. Let $l\in \mathbb{N}_{>0}$ be prime to $\mathrm{ch}(k(x))$. Then the Gersten conjecture for Milnor $K$-theory with mod $l$-coefficients holds in the Nisnevich topology. This follows from the Gersten conjecture for \'etale cohomology with torsion coefficients prime to the residue characteristic and the following sequence of rigidity and Bloch-Kato isomorphisms: 
$$K^M_n((\O_{X,x}^h))/l\cong K^M_n(k(x))/l\cong H^n_{\et}(k(x),\Z/l\Z(n))\cong H^n_{\et}((\O_{X,x}^h),\Z/l\Z(n)).$$ 
For a proof of the rigidity of Milnor $K$-theory see \cite[Prop. 3.7]{LudersBlochOgus}.
%Dahlhausen \cite[Lem. 3.2]{Da18}. 
%The rigidity for motivic cohomology is shown by Geisser in \cite[Thm. 1.2(3)]{Ge04}. 
The same result also holds with coefficients which are not prime to $\mathrm{ch}(k(x))$. This is much more subtle though and is the main result of the author and Morrow in \cite{LuMo20}.
\end{enumerate}
\end{remark}

%\subsection*{Funding.} The author is supported by the DFG Research Fellowship LU 2418/1-1.

\subsection*{Acknowledgements.} I would like to thank Matthew Morrow for pointing out the first part of Remark \ref{remark_introduction}(4), explaining facts about the DVR case to me and many helpful discussions on the Gersten conjecture, also in the context of \cite{LuMo20}. I would like to thank Christian Dahlhausen for discussions on Section \ref{Gerstenreldimzero}. %Finally I would like to thank an anonymous referee of IMNR for their careful reading and helping me improve the text considerably.

\section{Milnor $K$-theory, Chow groups with coefficients and geometric representation}

%\subsection{Milnor K-theory}
%We recall some results about Milnor K-theory which we will use in the proof of our main theorem.

%\begin{definition}
%A functor $F$ is called continuous if for every filtering direct limit of rings
%$$A=\varinjlim A_i$$
%the natural homomorphism
%$$\varinjlim F(A_i)\r F(A)$$
%is an isomorphism.
%\end{definition}

%\begin{proposition}(\cite[Prop. 6]{Ke10})
%The Milnor K-sheaf $\mathcal{K}^M_*$ is continuous.
%\end{proposition}

%Let $X$ be an excellent scheme and let $\mathcal{K}^M_{n,X}$ be the improved Milnor K-sheaf defined in \cite{Ke10}. 
%\begin{definition}\label{Gcdef} We say that the Gersten conjecture holds for the (Milnor K-)sheaf $\mathcal{K}^M_{n,X}$ if the sequence of sheaves
%$$0\r\mathcal{K}^M_{n,X}\rightarrow \bigoplus_{x\in X^{(0)}}i_{x,*}K^M_n(k(x))\rightarrow %\bigoplus_{x\in X^{(1)}}i_{x,*}K^M_{n-1}(k(x))\r...$$
%is exact.
%\end{definition} 
\subsection{Nisnevich descent and $\mathbb{A}^1$-invariance for Milnor $K$-theory}
We adopt the usual definition of Milnor $K$-theory, even for non-local rings:

\begin{definition}[Milnor $K$-theory]
Let $R$ be a (always commutative) ring. We define the $n^\mathrm{th}$ Milnor K-group $K^M_n(R)$ to be the quotient of $(R^{\times})^{\otimes n}$ by the Steinberg relations, i.e. the subgroup of $(R^{\times})^{\otimes n}$ generated by elements of the form $a_1\otimes\cdots\otimes a_n$ where $a_l+a_k=1$ for some $1\leq l<k\leq n$. As usual, the image of $a_1\otimes\cdots\otimes a_n$ in $K^M_n(R)$ is denoted by $\{a_1,\dots,a_n\}$.
\end{definition}

The Milnor $K$-theory sheaf is the sheafification in the Zariski topology of the above presheaf of abelian groups on the category of affine schemes. For a scheme $X$ we denote the restriction of the Milnor $K$-theory sheaf to the Zariski site of $X$ by $\mathcal K^M_{n,X}$. 
Associated to this sheaf, there exists an improved Milnor $K$-theory sheaf $\hat{\mathcal{K}}_{n,X}^M$ which is a modification the usual Milnor $K$-theory sheaf if the the residue fields of $X$ are small. In contrast to the original Milnor $K$-theory sheaf it has a transfer (also called norm map) for all finite \'etale extensions of local rings, not just if the residue fields are infinite. If $R$ is a local ring, then there is a surjection $K^M_n(R)=\mathcal K^M_{n,\Spec(R)}(R)\twoheadrightarrow \hat{\mathcal{K}}_{n,\Spec(R)}^M(R)=:\hat{K}^M_n(R)$ and this surjection is an isomorphism if the residue field of $R$ is infinite. The two groups coincide if $R$ is a field. 
The Gersten conjecture for Milnor $K$-theory says that for any regular local ring $R$, and $X=\Spec(R)$ of dimension $d$, the complex
$$0\to \hat{K}^M_n(R)\to K^M_n(\mathrm{Frac}(R))\r \bigoplus_{x\in X^{(1)}}K^M_{n-1}(x)\r ...\r \bigoplus_{x\in X^{(d)}}K^M_{n-d}(x)\r 0$$
is exact. Here $K^M_n(x):=K^M_n(\kappa(x))$. One expects the improved Milnor $K$-theory sheaf to satisfy the Gersten conjecture, but in general not the usual Milnor $K$-theory sheaf if the residue field of $R$ small. For more details see \cite{Ke10}, in particular for the precise sense of small we refer to \cite[Prop. 10(5)]{Ke10}.

In the axiomatic approach to the Gersten conjecture of Colliot-Th\'el\`ene-Hoobler-Kahn in \cite{CHK97} there are two main steps: Nisnevich descent and homotopy invariance. The proofs of the main results of Section \ref{section_5} may be interpreted as proceeding in an analogous fashion. For these proofs we need to recall two theorems which are due to Kerz and which were also used in this way in his proof of the Gersten conjecture for Milnor $K$-theory in equal characteristic in \cite{Ke09}: Milnor $K$-theory satisfies descent for Nisnevich-distinguished squares (Thm. \ref{Nisdescent}) and there is a generalisation of the Milnor-Bass-Tate exact sequence (Thm. \ref{milnorbasstate_fields}) to local rings (Thm. \ref{milnorbasstate}). The latter is central after having used Nisnevich descent (and geometric representation; see Section \ref{section_geom_rep}) to reduce the Gersten conjecture to affine space.

\begin{theorem}(\cite[Thm. 3.1]{Ke09})\label{Nisdescent}
Let $A\subset A'$ be a local ring extension of semi-local rings with infinite residue fields (i.e. $\Spec(A')\r \Spec(A)$ is dominant, maps closed points to closed points and is a surjective map on closed pints). Let $0\neq f,f_1\in A$ be such that $f_1|f$ and $A/(f)\cong A'/(f).$ Then the diagram
\[
  \xymatrix{ 
  K_n^M(A_{f_1}) \ar[r] \ar[d] & K_n^M(A_{f}) \ar[d] \\
  K_n^M(A'_{f_1}) \ar[r] & K_n^M(A'_{f}) 
  }
\]
is co-cartesian. 
\end{theorem}

Milnor, Bass, and Tate proved the following theorem.
\begin{theorem}(\cite{Milnor1970}, \cite{BassTate1973})\label{milnorbasstate_fields}
Let $F$ be a field. There exists a split exact sequence 
$$0\r K^M_n(F)\r K^M_n(F(t))\r \bigoplus_{\pi}K^M_{n-1}(F[t]/(\pi))\r 0$$
where the direct sum is over all monic irreducible $\pi\in F[t]$.
\end{theorem}

In order to formulate the generalisation to local rings by Kerz we need to introduce some notation and definitions. In the following let $A$ be a semi-local domain which is factorial and has infinite residue field. Let $F$ be the quotient field of $A$. 
\begin{definition}\label{definition_feasible_tuples}
\begin{enumerate}
\item An $n$-tuple of rational functions 
$$(p_1/q_1,p_2/q_2,...,p_n/q_n)\in F(t)^n$$
with $p_i,q_i\in A[t]$ and $p_i/q_i$ a reduced fraction for $i=1,...,n$ is called feasible if the highest non vanishing coefficients of $p_i,q_i$ are invertible in $A$ and if for any irreducible factors $u$ of $p_i$ or $q_i$ and $v$ of $p_j$ or $q_j$ $(i\neq j)$, we have that $u=av$ with $a\in A^*$ or $(u,v)=1$. 
\item Let $$\mathcal{T}^t_n(A):=\Z<(p_1,...,p_n)|(p_1,...,p_n) \:feasible, p_i\in A[t]\: irreducible\: or\: unit>/Linear$$
Here $Linear$ denotes the subgroup generated by elements 
$$(p_1,...,ap_i,...,p_n)-(p_1,...,a,...,p_n)-(p_1,...,p_i,...,p_n)$$
with $a\in A^*$.

By bilinear factorization the element 
$$(p_1,...,p_n)\in \mathcal{T}^t_n(A)$$
is defined for every feasible $n$-tuple with $p_i\in F(t)$.

\item Let $$St\subset \mathcal{T}^t_n(A)$$ be the subgroup generated by feasible $n$-tuples $(p_1,...,p,p-1,...,p_n)$ and $(p_1,...,p,-p,...,p_n)$ with $p_i,p\in F(t)$.
\item Set $$K^t_n(A):=\mathcal{T}^t_n(A)/St.$$
\item Let $0\neq p\in A[t]$ be an arbitrary monic polynomial. Define 
$$K^t_n(A,p)$$
analogously to $K^t_n(A)$ but this time an $n$-tuple $(p_1/q_1,p_2/q_2,...,p_n/q_n)$ is feasible if additionally all $p_i,q_i$ are coprime to $p$.\footnote{One should think of $K^t_n(A)$ as the fraction field of the more local object $K^t_n(A,p)$. This will be important in the proof of Theorem \ref{thm1intext} \& \ref{thm2intext}.}
\end{enumerate}
\end{definition}

Then the following generalization of the Milnor-Bass-Tate sequence holds:
\begin{theorem}(\cite[Thm. 4.5]{Ke09})\label{milnorbasstate}
There exists a split exact sequence 
$$0\r K^M_n(A)\r K^t_n(A,p)\r \bigoplus_{\pi}K^M_{n-1}(A[t]/(\pi))\r 0$$
where the direct sum is over all monic irreducible $\pi\in A[t]$ with $(\pi,p)=1$.
\end{theorem}

The two theorems are compatible in the following sense:
\begin{proposition}(\cite[Prop. 4.7]{Ke09})
Let the notation be as above and $F=\mathrm{Frac}(A)$. There is a commutative diagram with exact rows
\[
  \xymatrix{ 
 0 \ar[r] & K^M_n(A) \ar[r] \ar[d]  & K^t_n(A,p) \ar[r] \ar[d] & \bigoplus_{\pi}K^M_{n-1}(A[t]/(\pi)) \ar[r] \ar[d] & 0 \\
 0 \ar[r] & K^M_n(F) \ar[r]  & K^M_n(F(t)) \ar[r] & \bigoplus_{\pi}K^M_{n-1}(F[t]/(\pi)) \ar[r] & 0 
  }
\]
in which the sum in the upper row is taken over all monic irreducible $\pi\in A[t]$ with $(\pi,p)=1$ and the sum in the lower row is taken over all monic irreducible $\pi\in F[t]$.
\end{proposition}

\subsection{Chow groups with coefficients}\label{subsection_chow_groups_with_coeff}
We introduce Chow groups with coefficients from \cite{Ro96} and recall some of their properties. These will be used in the next section to prove the exactness of the Gersten complex on the right. Note that in \cite{Ro96} Rost assumes that the base is a field. We need his results concerning the four basic maps (Section 4 in \textit{loc. cit.}, more precisely Lemmas \ref{lemma_correspondences_compatibility} and \ref{lemma_Rost45} and their references in our text below) for schemes over discrete valuation rings but we formulate them in the largest generality possible. 
In this section let $S$ be an excellent\footnote{In order for the Gersten complex of a noetherian scheme $X$ to be a complex one needs to assume $X$ to be excellent (see \cite{Ka83} for the case of varieties over a field and \cite[Sec. 8.1]{GS06} for the more general case of excellent noetherian schemes). Throughout the article we assume that the base scheme $S$ is excellent which implies that all the schemes which we consider and which are of finite type over $S$ are excellent.} noetherian scheme and $X$ be a $d$-dimensional scheme of finite type over $S$. Let
$$C_p(X,n):= \bigoplus_{x\in X_{(p)}}K^M_n(x)$$
and let 
$$d=d_X=\oplus \partial^x_y:C_p(X,n)\xrightarrow{} C_{p-1}(X,n-1)$$
be the map induced by the tame symbol (see \cite[(3.2)]{Ro96}). It can be shown that these maps fit into a complex
$$C_{d-*}(X,n-*):= C_d(X,n)\xrightarrow{d} C_{d-1}(X,n-1)\xrightarrow{d} C_{d-2}(X,n-2)\r ...$$
which we call cycle complex or Gersten complex for Milnor $K$-theory. If $n$ is arbitrary we also write $A^p(X)$ for $H^p(C_{d-*}(X,n-*))$ and $A_p(X)$ for $H^{d-p}(C_{d-*}(X,n-*))$.

There are the following four standard (also called basic) maps:
\begin{definition}\label{definition_correspondences}
Let $X$ and $Y$ be schemes of finite type over $S$.
\begin{enumerate}
\item (\cite[(3.4)]{Ro96}) Let $f:X\r Y$ be an $S$-morphism of schemes. Then the pushforward
$$f_*:C_p(X,n)\r C_p(Y,n)$$
is defined as follows:
  \[
    (f_*)^x_y= \left\{
                \begin{array}{ll}
                  N_{k(x)/k(y)} \quad \mathrm{if}\quad k(x)\quad \mathrm{finite} \quad \mathrm{over}\quad k(y),\\
                  0 \quad \quad \quad \quad\quad \mathrm{otherwise}.
                \end{array}
              \right.
  \]
\item (\cite[(3.5)]{Ro96}) Let $g:Y\r X$ be an equidimensional $S$-morphism of relative dimension $s$. Then 
$$g^*:C_p(X,n)\r C_{p+s}(Y,n)$$
is defined as follows
  \[
    (g^*)^x_y= \left\{
                \begin{array}{ll}
                  \mathrm{length}(\O_{Y_x,y})i_* \quad \mathrm{if}\quad \mathrm{codim}_y(Y_x)=0,\\
                  0 \quad \quad \quad \quad\quad \mathrm{otherwise}.
                \end{array}
              \right.
  \]
Here $i:k(x)\r k(y)$ is the inclusion on residue fields and $i_*$ the induced map on Milnor $K$-theory.
\item (\cite[(3.6)]{Ro96}) Let $a\in \O^*_X(X)$. Then 
$$\{a\}:C_p(X,n)\r C_p(X,n+1)$$
is defined by
$$\{a\}^x_y(\rho)=\{a\}\cdot \rho$$
for $x=y$ and zero otherwise.

\item (\cite[(3.7)]{Ro96})  Let $Y$ be a closed subscheme of $X$ and $U=X\setminus Y$. Then the boundary map 
$$\partial^U_Y:C_p(U,n)\r C_{p-1}(Y,n-1)$$
is defined by the tame symbol $\partial^x_y$.
  
\end{enumerate}
\end{definition}

\begin{notat*}
We denote by 
$$\alpha:X\Rightarrow Y$$
a homomorphism $\alpha:C_{d-*}(X,n-*)\r C_{d'-*}(X,m-*)$ which is the composition of basic maps defined in Definition \ref{definition_correspondences}. $\alpha:X\Rightarrow Y$ is called a correspondence. 
%between cycle modules and sometimes abbreviate the cycle modules $C_p(X,n)$ associated to $X$ by $X$.
\end{notat*} 
The next lemma shows that a correspondence is compatible with the boundary maps $d_X$ and $d_Y$. The lemma is formulated in \cite[Prop. 4.6]{Ro96} for schemes over fields but the proofs transfer almost word by word to our situation. Indeed, either the steps of the proof reduce to equal characteristic or to statements which are known for arbitrary discrete valuation rings. We give some details.

\begin{lemma}(\cite[Prop. 4.6]{Ro96})\label{lemma_correspondences_compatibility}
Let $X$ and $Y$ be schemes of finite type over $S$.
\begin{enumerate}
\item Let $f:X\r Y$ be proper $S$-morphism. Then $$d_Y\circ f_*=f_*\circ d_X.$$
\item Let $g:Y\r X$ be flat $S$-morphism which is equidimensional with constant fibre dimension. Then
$$g^*\circ d_X=d_Y\circ g^*.$$
\item Let $a\in\O_X(X)^{\times}$. Then
$$d_X\circ \{a\}=-\{a\}\circ d_X.$$
\end{enumerate}
\end{lemma} 
\begin{proof}
(1) Let $\delta(f_*)=d_Y\circ f_*-f_*\circ d_X$. We need to show that $\delta(f_*)^x_y=0$ for $x\in X_{(p)}$ and $y\in Y_{(p-1)}$. Let $z=f(x)$. 

If $y\notin \overline{\{z\}}$, then the claim is clear since both boundary maps are zero. 

If $y=z$, then we replace $Y$ by $\Spec(k(y))$ and $X$ by $\overline{\{x\}}\cap X_s$, where $s$ is the image of $y$ in $S$ under the structure morphism $Y\to S$. This is then the case of a proper curve over a field in equal characteristic which is proved in \cite[Prop. 7.4.4]{GS06}. 

If  $y\notin \overline{\{z\}}$ and $y\neq z$, then we may assume that $Y=\overline{\{z\}}$ and $X=\overline{\{x\}}$ and dim$X=\mathrm{dim} Y$. Consider the commutative diagram 
\[
  \xymatrix{ 
  \tilde{X}\ar[r]^g \ar[d]_{\tilde{f}} & X \ar[d]^f \\
  \tilde{Y} \ar[r]^h & Y
  }
\]
in which $g$ and $h$ are the normalisations. Let $\tilde{x}$ be the generic point of $\tilde{X}$ and $\tilde{z}$ be the generic point of $\tilde{Y}$. Then $\delta(\tilde{g}_*)=\delta(\tilde{h}_*)=0$ (which holds by the definition of the boundary map) implies that 
$$\delta(f_*)^x_y\circ g_*=d_Y\circ f_*\circ g_*-f_*\circ d_X\circ g_*=
d_Y\circ h_*\circ \tilde{f}*-f_*\circ g_*\circ d_{\tilde X}=
h_*\circ d_{\tilde{Y}}\circ \tilde{f}*-h_*\circ \tilde{f}*\circ d_{\tilde X}=
h_*\circ\delta(\tilde{f}_*)^{\tilde{x}}_{\tilde{y}}$$
as maps from $K^M_n(\tilde x)$ to $K^M_{n-1}(y)$.
Since $g_*:K^M_n(\tilde x)\r K^M_n(x)$ is an isomorphism, it suffices to show that $\delta(\tilde{f}_*)^{\tilde{x}}_{\tilde{y}}=0$ in order to show that $\delta(f_*)^x_y=0$. But this follows from \cite[Cor. 7.4.3]{GS06} since all points $\tilde{u}\in \tilde{X}$ over $\tilde{y}$ correspond to discrete valuations extending the discrete valuation induced by $\tilde y$.

(2) Let $\delta(g^*)=d_Y\circ g^*-g^*\circ d_X$. We need to show that $\delta(g^*)^x_y=0$ for $x\in X_{(p)}$ and $y\in Y_{(p+s-1)}$. Let $z=g(y)$.

If $z\notin \overline{\{x\}}$, then the claim is clear.

If $z=x$, then let $u\in Y_x$ be a point specialising to $y$. Then the composition $K^M_n(x)\rightarrow K^M_n(u)\xrightarrow{\partial}K^M_{n-1}(y)$ is zero since the valuation on $k(u)$ induced by $y$ is trivial on $k(x)$.

For the case of $z\in \overline{\{x\}}$ and $z\neq x$ we refer to \textit{loc. cit.}.

(3) Follows immediately from the definitions.
\end{proof}

We will also need the following lemma:
\begin{lemma}(\cite[Lemma 4.5]{Ro96})\label{lemma_Rost45}
Let $X$ and $Y$ be schemes of finite type over $S$. Let $g:Y\r X$ be a smooth $S$-morphism of constant fibre dimension $1$ and let $\sigma:X\r Y$ be a section to $g$ and $t\in \O_Y$ be a global parameter defining the subscheme $\sigma(X)$. Moreover, let $\tilde{g}:Y\setminus \sigma(X)\r X$ be the restriction of $g$ and let $\partial$ be the boundary map associated to $\sigma$. Then $$\partial_X\circ \tilde{g}^*=0$$
and   
$$\partial_X\circ  \{t\}\circ\tilde{g}^*=((\mathrm{id})_X)_*.$$
\end{lemma}
\begin{proof}
Both statements reduce to the case that $X$ is the spectrum of a field and then follow from the definitions of the tame symbol.
\end{proof}

\subsection{Geometric presentation}\label{section_geom_rep}\subsubsection{Hyperplanes and linear projections}
We recall some facts about hyperplanes and linear projections in projective space over a discrete valuation ring from \cite[Sec. 0]{Jannsen2012}.

If $L$ is a field, then the dual projective space $(\mathbb{P}_L^N)^\vee$ parametrizes hyperplanes in $\mathbb{P}_L^N$, an $L$-rational point $(a_0:...:a_N)$ corresponding to the hyperplane given by the equation $$a_0x_0+...+a_Nx_N=0,$$ the $x_i$ being the homogeneous coordinates of $\mathbb{P}_L^N$. If $L$ is infinite, then $\mathbb{P}_L^N(L)$ is Zariski dense in $\mathbb{P}_L^N$.

Let $A$ be a discrete valuation ring with function field $K$ and infinite residue field $k$. Then again a hyperplane $H\subset \mathbb{P}_A^N$ is given by an $A$-rational point of the dual projective space $(\mathbb{P}_A^N)^\vee$, or, equivalently, an equation $$a_0x_0+...+a_Nx_N=0, a_i\in A,$$ such that not all $a_i$ lie in the maximal ideal of $A$.

\begin{proposition}\label{proposition_hyperplanes}
Let $V_k$ be a Zariski open dense subset of $(\mathbb{P}_k^N)^\vee$ and $V_K$ be a Zariski-open dense subset of $(\mathbb{P}_K^N)^\vee$. Then there exists a hyperplane $H\subset \mathbb{P}_A^N$ whose restriction to $\mathbb{P}_K^N$, denoted by $H_\eta$, lies in $V_K$ and whose restriction to $\mathbb{P}_k^N$, denoted by $H_s$, lies in $V_k$.
\end{proposition}
\begin{proof}
Let $$sp:(\mathbb{P}_K^N)^\vee(K)\to (\mathbb{P}_k^N)^\vee(k)$$ be the specialisation map which sends $H_\eta$ to $H_s$. Since a hyperplane $H\subset \mathbb{P}_A^N$ is completely determined by $H_\eta$, we need to show that the intersection  $V_K(K)\cap sp^{-1}(V_k(k))$ is non-empty. Let $Z_K=(\mathbb{P}_K^N)^\vee\setminus V_K$ and $Z_k=(\mathbb{P}_k^N)^\vee\setminus V_k$. Let $sp(Z_K)=\overline{Z_K}\cap \mathbb{P}_k^N$, $\overline{Z_K}$ being the Zariski closure of $Z_K$ in $(\mathbb{P}_A^N)^\vee$. The properness of $\mathbb{P}_A^N$ over $\Spec(A)$ implies that every $K$-rational point of the generic fibre specialises to a $k$-rational point of the special fibre and therefore $Z_K(K)\subset sp^{-1}(sp(Z_K)(k))$. This implies that $sp^{-1}((V_k\setminus sp(Z_K))(k)) \subset V_K(K)\cap sp^{-1}(V_k(k))$. But the set $sp^{-1}((V_k\setminus sp(Z_K))(k)) $ contains a $K$-rational point since $sp:(\mathbb{P}_K^N)^\vee(K)\to (\mathbb{P}_k^N)^\vee(k)$ is surjective and $V_k\setminus sp(Z_K)$ contains a $k$-rational point. Therefore $V_K(K)\cap sp^{-1}(V_k(k))$ contains a $K$-rational point.
\end{proof}
\subsubsection{Noether normalization over a DVR}
We will need the following version of Noether normalization over a DVR in the proof of Theorem \ref{thmrightexactintext}. 
\begin{lemma} (\cite[Lem. 1]{GL87})\label{geomrep1} Let $X$ be an affine scheme flat of finite type over a discrete valuation ring $\O_K$. Let $P\subset X$ be a finite set of points such that $X$ is smooth at $P$ over $\O_K$. Let $Y$ be a principal effective divisor which is flat over $\O_K$.
Then there exists an affine open neighbourhood $X_0\subset X$ of $P$ and a morphism $\pi:X_0\r \mathbb{A}^{d-1}_{\O_K}$, where $d$ is the relative dimension of $X$ over $\Spec(\O_K)$, such that
\begin{enumerate}
\item $\pi$ is smooth of relative dimension $1$ at the points of $P$.
\item $\pi$ restricted to $Y_0=X_0\cap Y$ is quasi-finite.
\end{enumerate}
\end{lemma}

Even though we do not need the result in this article, we note that the above lemma can be generalised to higher dimensional bases if $P$ is a finite set of closed points.

\begin{lemma} Let $S$ be the spectrum of a local ring $A$. Let $X$ be an affine scheme flat of finite type of relative dimension $d$ over $S$. Let $P\subset X$ be a finite set of closed points such that $X$ is smooth at $P$ over $S$. Let $Y$ be a principal effective divisor which is flat over $S$.
Then there exists an affine open neighbourhood $X_0\subset X$ of $P$ and a morphism $\pi:X_0\r \mathbb{A}^{d-1}_{S}$ such that
\begin{enumerate}
\item $\pi$ is smooth of relative dimension $1$ at the points of $P$.
\item $\pi$ restricted to $Y_0=X_0\cap Y$ is quasi-finite.
\end{enumerate}
\end{lemma}
\begin{proof}
Let $k$ be the residue field of $A$. 
%Let $\bar X$ be the closure of $X$ in some projective space $\mathbb P^M_S$ and $\bar{Y}$ be the closure of $Y$ in $\bar X$. Let $\bar X_k$ and $\bar Y_k$ denote the special fibers of $\bar X$ and $\bar Y$ over $S$. For every $p\in P$ let $s$ be a specialisation of $p$ in $\bar X_k$, that is for example a closed point of $\overline{\{p\}}\cap \bar X_k$. Let $\bar P$ be the set of all such $s$. 
%Let $U=\Spec(B)$ be an affine neighbourhood of $P$. 
By \cite[Sec. 7, 5.12]{Qu72} there exists a morphism $\pi:X_k=\Spec(B\otimes_A k)\to \mathbb A^d_k=\Spec(k[x_1,...,x_d])$ which is smooth in a neighbourhood of $P$ and finite when restricted to $\bar Y$. Let $\bar x_1,...,\bar x_d$ be the images of $x_1,...,x_d$ in $B\otimes_A k$ and $\tilde x_1,...,\tilde x_d$ be lifts to $B$. These lifts induce a morphism $\pi:X\to \mathbb A_S^d$. Since $X\to S$ is flat, \cite[II, Cor. 2.2]{SGA1} implies that $\pi$ is smooth in an open neighbourhood of $P$. By \cite[Tag 01TI]{stacks-project} the quasi-finite locus of a morphism of schemes is open. This implies that there is an open subset $V\subset Y$ containing $P\cap Y$ such that $\pi\mid_V$ is quasi-finite. Then we can set $X_0=X-(Y-V)$.
\end{proof}

\subsubsection{Gabber's geometric presentation theorem over a DVR}
%In Section \ref{section_5} we need the following geometric representation theorem:
We recall Gabber's geometric presentation theorem over a discrete valuation ring from \cite[Thm. 2.4]{SS18}:
\begin{proposition}\label{proposition_Strunk_Schmidt}
	Let $S$ be the spectrum of a henselian discrete valuation ring $\mathcal O_K$ with infinite residue field $k$. Let $s$ be the closed point of $S$. Let $X$ be a smooth affine $S$-scheme of finite type, fibrewise of pure dimension $n$ and let $Z\hookrightarrow X$ be a proper closed subscheme. Let $z$ be a point in $Z$. If $z$ lies in the special fibre, suppose that $Z_s\neq X_s$. Then, Nisnevich-locally around $z$, there exists a closed embedding $X\hookrightarrow \mathbb A^N_S$ and a linear projection $$p_{\underline{u}}=p_{(u_1,...,u_{n-1})}\times _S p_{u_n}:X\to  \mathbb A^{n-1}_S\times\mathbb A^1_S$$
	and Zariski-open neighbourhoods $p_{(u_1,...,u_{n-1})}(z)\in V \subset  \mathbb A^{n-1}_S$ and $z\in U\subset p_{(u_1,...,u_{n-1})}^{-1}(V)$ such that
	\item[(1)] $Z\cap U=Z\cap p_{(u_1,...,u_{n-1})}^{-1}(V)$,
	\item[(2)] $p_{(u_1,...,u_{n-1})}|_Z:Z\to \mathbb A^{n-1}_S$ is finite,
	\item[(3)] $p_{\underline u}|_U:U\to \mathbb A^{n}_S$ is \'etale and restricts to a closed immersion $Z\cap U\hookrightarrow \mathbb A^1_V$, and
	%\item[(4)] $\varphi(z_i)\notin \varphi(Z)$ if $z_i\notin Z$,
	\item[(4)] $p_{\underline u}^{-1}(p_{\underline u}(Z\cap U))\cap U=Z\cap U$.
\end{proposition}

In Section \ref{section_5} we need the following Zariski-local version. The difference is that we drop the finiteness of the map $p_{(u_1,...,u_{n-1})}|_Z:Z\to \mathbb A^{n-1}_S$ and replace it by quasi-finiteness (as in Lemma \ref{geomrep1}) which allows us to work Zariski-locally. The properties of the proposition that we need for the proofs in Section \ref{section_5} are (3) and (4).

\begin{proposition}\label{geomrep2}%\label{geom_pres_zariski}
	Let $S$ be the spectrum of a discrete valuation ring $\mathcal O_K$ with infinite residue field $k$. Let $s$ be the closed point of $S$. Let $X$ be a smooth affine irreducible $S$-scheme of finite type, fibrewise of pure dimension $n$ and let $Z\hookrightarrow X$ be a proper closed subscheme. Let $z_1,...,z_t$ be closed points in $Z$. Suppose that the special fiber $Z_s$ of $Z$ does not contain any irreducible component of $X_s$. Then, Zariski-locally around $z$, there exists an immersion $X\hookrightarrow \mathbb A^N_S$ and a linear projection $$p_{\underline{u}}=p_{(u_1,...,u_{n-1})}\times _S p_{u_n}:X\to  \mathbb A^{n-1}_S\times\mathbb A^1_S$$
	and Zariski-open neighbourhoods $p_{(u_1,...,u_{n-1})}(\{z_1,...,z_t\})\in V \subset  \mathbb A^{n-1}_S$ and $z_1,...,z_t\in U\subset p_{(u_1,...,u_{n-1})}^{-1}(V)$ such that
	\item[(1)] $Z\cap U=Z\cap p_{(u_1,...,u_{n-1})}^{-1}(V)$,
	\item[(2)] $p_{(u_1,...,u_{n-1})}|_Z:Z\to \mathbb A^{n-1}_S$ is quasi-finite,
	\item[(3)] $p_{\underline u}|_U:U\to \mathbb A^{n}_S$ is \'etale and restricts to a closed immersion $Z\cap U\hookrightarrow \mathbb A^1_V$, and
	%\item[(4)] $\varphi(t_i)\notin \varphi(Z)$ if $t_i\notin Z$,
	\item[(4)] $p_{\underline u}^{-1}(p_{\underline u}(Z\cap U))\cap U=Z\cap U$.
\end{proposition}
\begin{proof}
Choose an embedding $X\hookrightarrow \mathbb A^N_S$ and take the closure $\bar X$ of $X$ inside of $\mathbb P^N_S$. Let $\bar Z$ be the closure of $Z$ in $\bar X$. Let $H$ be a hyperplane in $\mathbb P^N_S$ which does not contain $\bar X$ and which has the property that $z_1,...,z_t\notin H$ and that no irreducible component of $\bar Z_s$ is contained in $H_s$. The existence of such a hyperplane follows from Bertini theorems and Proposition \ref{proposition_hyperplanes}. Set $X':= \bar X \setminus H$ and $Z':= X'\cap \bar Z$. Then $X'$ is a closed subscheme of  $\mathbb A^{N}_S=\mathbb P^{N}_S\setminus H$ and $Z'$ is fibrewise dense in the closure of $Z'$ in $\mathbb P^{N}_S$. This is the key property which is needed for the proof of Proposition \ref{proposition_Strunk_Schmidt} and which is achieved in \cite{SS18} by working Ninevich-locally (see the outline of the proof after Lemma 2.3 and Proposition 2.6 of \textit{loc. cit.}). We can now follow the proof in \cite{SS18} from paragraph 2.8 onwards noticing that even though $X'$ is not necessarily smooth over $\Spec(\mathcal O_K)$, it suffices that it is still smooth at $z_1,...,z_t$. Note that we can work with more than one point $z_i\in Z$ since the sets of linear projections satisfying the proposition for one $z_i$ are Zariski-open subsets of the space of projections. Furthermore we do not need $\mathcal O_K$ to be henselian since the map $W(\sigma)\to W(\mathbb F)$ in the proof of Proposition 2.9 of \textit{loc. cit.} is always surjective as $W$ is a subset of affine space (this follows from the surjectivity of the specialisation map in the proof of Proposition \ref{proposition_hyperplanes}).

Finally we intersect $X'$ back with $X$ which gives us the proposition and is the reason for the quasi-finiteness, and not finiteness, in (2).
\end{proof}

\begin{remark}
%\begin{enumerate}
%\item
%\item 
For the proofs of Theorems \ref{thm1intext} and \ref{thm2intext} we do not need the quasi-finiteness property (2) in Proposition \ref{geomrep2} since it is not needed for the application of Theorem \ref{Nisdescent}, that is Nisnevich descent. 
For the proof of Theorem \ref{thm2intext} it would therefore suffice to use \cite[Thm. 3.14]{Dutta1995} of Dutta if we additionally assume that $k$ is perfect. For the proof of Theorem \ref{thm1intext} one needs to work with semi-local rings in the proof. We do not know if the proof of \cite[Thm. 3.14]{Dutta1995} can be modified to cover that situation, therefore we use Proposition \ref{geomrep2}. 
%\end{enumerate}
\end{remark}

\section{Exactness on the right}
In this section we prove the following theorem. The proof is a combination of the strategies of the proofs of \cite[Thm. 6.1]{Ro96} and \cite[Sec. 2]{GL87}. In \cite{Ro96} the following theorem is proved for smooth schemes over a field and is attributed to Gabber. The essential ingredient is ``Quillen's trick'' which in \cite{GL87} is adapted to the relative case for algebraic $K$-theory. This ``trick'', originally used by Quillen in \cite{Qu72} to prove the Gersten conjecture for algebraic $K$-theory for smooth schemes over a field, makes a reduction to the case of divisors and the main tool in this reduction is a version of Noether normalisation. Another important ingredient in the proof, which is needed when using  ``Quillen's trick'' in the relative situation, is Zariski's main theorem.
\begin{theorem}\label{thmrightexactintext}
Let $X$ be the spectrum of a smooth local ring over an excellent discrete valuation ring $\O_K$. Let $\dim X=d+1$ or in other words let $X$ be of relative dimension $d$ over $\Spec(\O_K)$. Then 
$$A^p(X)=0$$
for $p>0$. In other words the complex
$$\bigoplus_{x\in X^{(0)}}K^M_n(x)\r \bigoplus_{x\in X^{(1)}}K^M_{n-1}(x)\r ...\r \bigoplus_{x\in X^{(d+1)}}K^M_{n-d-1}(x)\r 0$$
is exact (see Section \ref{subsection_chow_groups_with_coeff}).
\end{theorem}
We need the following lemma:
\begin{lemma}\label{mainlemma}
Let $X$ be a smooth scheme of $X$ of relative dimension $d$ over an excellent discrete valuation ring $\O_K$. Let $Y$ be a closed subscheme of codimension $c>0$ which is flat over $\O_K$. Let $x\in Y$ be a point. Then there is neighbourhood $U\subset X$ of $x$ such that $$A_p(Y\cap U) \xrightarrow{i_*} A_p(U)$$
is zero for all $p$, where $i:Y\cap U\r U$ is the inclusion.
\end{lemma}
Note that $Y$ being flat over $\O_K$ is equivalent to each component of $Y$ mapping surjectively to $\Spec(\O_K)$ (see \cite[Prop. 4.3.9]{Li02}). The statement implies in particular that the composition $A_p(Y)\to A_p(X)\to A_p(U)$ is the zero map.
\begin{proof}
By Lemma \ref{geomrep1}, replacing $X$ by a smaller affine neighbourhood we may find a morphism 
$$\pi:X\r \mathbb{A}_{\O_K}^{d-1}$$
of relative dimension $1$ which is smooth at $x$ and such that $\pi|_Y$ is quasi-finite. Note that $Y$ is of arbitrary codimension $c>0$ and that Lemma \ref{geomrep1} assumes that $Y$ is a principal effective divisor. However, since $X$ is regular, if codim$_X(Y)=1$, then $Y$ is locally principal and if codim$_X(Y)>1$, then we can take any element $f$ in the maximal ideal of $\O_{X,Y}$ (the stalk of $\O_X$ at the generic point of $Y$) with trivial valuation at the generic point of the special fibre of $X$ and for this element there is an open neighbourhood $U$ of $x$ in which $V(f)$ contains $U\cap Y$ and satisfies the conditions of Lemma \ref{geomrep1}. We then apply Lemma \ref{geomrep1} to $U$ and $V(f)$.

The inclusion $i:Y\r X$ induces a section $\sigma: Y\r Y\times_{\mathbb{A}_{\O_K}^{d-1}}X=: Z$ to the projection $q: Z:=Y\times_{\mathbb{A}_{\O_K}^{d-1}}X\r Y$, i.e.  there is commutative diagram

\[
  \xymatrix{ 
  Y \ar@{^{(}->}[rrd]^i \ar[ddr] \ar[dr]^\sigma & & \\
   & Z\ar[r]^p \ar[d]^q  & X \ar[d]^\pi
  \\
 & Y   \ar[r]^{\pi|_Y} & \mathbb{A}_{\O_K}^{d-1}.
  }
\]

By Zariski's main theorem (see \cite[18.12.13]{EGA4IV}) the quasi-finite morphism $p$ factors as an open immersion $Z\r \bar{Z}$ followed by a finite morphism $\bar{p}:\bar{Z}\r X$, i.e.  there is commutative diagram

\[
  \xymatrix{ 
  Y \ar@{^{(}->}[rrd]^i \ar[ddr] \ar[dr]^\sigma \ar[rr]^{\bar{\sigma}} & & \bar{Z} \ar[d]^{\bar{p}} &  D:=\bar{Z}\setminus Z\ar@{_{(}->}[l]_-{i'} \\
   & Z \ar[r]^p \ar[d]^q \ar[ur] & X \ar[d]^\pi &
  \\
 & Y   \ar[r]^{\pi|_Y} & \mathbb{A}_{\O_K}^{d-1}. &
  }
\]
Note that $\sigma(Y)$ is smooth over $Y$ and that $q$ is smooth at $p^{-1}(x)$. Therefore by \cite[Exp. II, Thm. 4.15]{SGA1}, $\sigma(Y)$ is regularly immersed at $p^{-1}(x)$.
Passing to a smaller open subscheme of $X$ we may assume that $\sigma(Y)\subset Z$ is given by a section $t\in \O_Z(Z)$. Let $D:=\bar{Z}\setminus Z$ and $\bar{\sigma}:Y\xrightarrow{\sigma}Z\r \bar{Z}$. Since $i=\bar{p}\circ \bar{\sigma}$ is a closed immersion, $\bar{\sigma}$ is a closed immersion (see \cite[Lem. 3.3.15]{Li02}). Furthermore $D\cap \bar{\sigma}(Y)=\emptyset$. By the Chinese remainder theorem, and modifying $X$ appropriately (see the proof of \cite[Thm. 4.1]{Ge04}), there is a section $t'\in\O_{\bar{Z}}(\bar{Z})$ which is $1$ on $D$ and whose zero set is $\bar{\sigma}(Y)$. In particular, $\mathrm{div}(t')=\bar{\sigma}(Y)$. Indeed, if $Z=\Spec(T)$ and $\bar{Z}=\Spec(\bar{T})$ and if $I\subset {T}$ defines ${\sigma}(Y)$ and $J\subset \bar{T}$ defines $D$, then by the Chinese remainder theorem there is a $t'\in \bar{T}$ mapping to $(1,t)\in \bar{T}/J\oplus \bar{T}/(I\bar{T})^2$. Note that since $D\cap \bar{\sigma}(Y)=\emptyset$ we have that $\bar{T}/(I\bar{T})^2\cong T/I^2$. Then we shrink $X$ by $\bar{p}(\mathrm{div}(t')\setminus \bar{\sigma}(Y))$. 

Let $Q:=Z\setminus \sigma(Y)$, $\bar{Q}:=\bar{Z}\setminus \bar{\sigma}(Y)$, $j:Q\r \bar{Q}$, $j':\bar{Q}\r \bar{Z}$ and $q':=q|_Q$. Consider the correspondence (Definition \ref{definition_correspondences})
$$H: Y\stackrel{q'^*}{\Longrightarrow} Q \stackrel{\{t'\}}{\Longrightarrow} Q \stackrel{j_*}{\Longrightarrow} \bar{Q}  \stackrel{j'_*}{\Longrightarrow} \bar{Z} \stackrel{\bar{p}_*}{\Longrightarrow} X$$
giving a homomorphism
$$H:C_p(Y,n)\r C_{p+1}(X,n+1).$$ Note that $q'^*$ shifts the degree by $(1,0)$ since $q'$ is of relative dimension $1$ and that $\{t'\}$ shifts the degree by $(0,1)$.
This gives a chain homotopy in the sense that
$$d_X\circ H+H\circ d_Y=i_*.$$
In other words, the map of chain complexes $i:C_{d+1-c-*}(Y,n-c-*)[-c]\r C_{d+1-*}(X,n-*)$ is null-homotopic and therefore the induced map on homology groups is zero. Indeed, 
\begin{align*}
 d_X\circ H &=d_X\circ \bar{p}_* \circ j'_*\circ j_*\circ \{t'\}\circ q'^* & (\mathrm{Def.\; of\;} H) \\
 &= \bar{p}_*\circ d_{\bar{Z}} \circ j'_*\circ j_*\circ \{t'\}\circ q'^* & (\mathrm{Lem.\;}\ref{lemma_correspondences_compatibility}(1)) \\
 &= \bar{p}_*\circ (j'_*\circ j_*\circ d_{Q}+i'_{D_*}\circ \partial^{{Q}}_{D}+\bar\sigma_*\circ \partial^{{Q}}_{\sigma(Y)}) \circ \{t'\}\circ q'^* & (\mathrm{Def.\; of\;} d \;\&\; \bar{Z}=Q\cup D\cup \sigma(Y))  \\
 &= \bar{p}_*\circ j'_*\circ j_*\circ d_{Q} \circ \{t'\}\circ q'^*+\bar{p}_*\circ \bar\sigma_*\circ \partial^{Q}_{\sigma(Y)} \circ \{t'\}\circ q'^* & (t'\mid_D=1\Rightarrow i'_{D_*}\circ \partial^{{Q}}_{D}\circ \{t'\}=0)\\
 &=- \bar{p}_*\circ j'_*\circ j_*\circ \{t'\}\circ d_{Q} \circ q'^*+\bar{p}_*\circ \bar\sigma_*\circ \mathrm{id}_* & (\mathrm{Lem.\;}\ref{lemma_correspondences_compatibility}(3) \;\&\;  \mathrm{Lem.\;}\ref{lemma_Rost45}) \\
 &= -H\circ d_Y +i_*. & (\mathrm{Def.\; of\;} H, \mathrm{Lem.\;}\ref{lemma_correspondences_compatibility}(2)\;\&\; p\circ\sigma=i) 
\end{align*} 
%Note that the fourth equality follows from the fact that $t'$ is equal to $1$ along $D$.
\end{proof}

\begin{proof}[Proof of Theorem \ref{thmrightexactintext}]
Let $(U,x)$ be a pair of a smooth scheme $U$ over $\O_K$ and a point $x\in U$ with $\Spec(\mathcal{O}_{U,x})=X$. Then
$$C_p(X,n)=\varinjlim_{(U,x)}C_p(U,n),$$
where the colimit is taken over all pairs $(U,x)$ as above.
Furthermore
$$C_p(U,n)=\varinjlim_{Y}C_p(Y,n)$$
where $Y$ runs through the closed subschemes of $U$ of dimension $p$. Therefore 
\begin{equation}\label{equation_proof_exactness_right}
A^{d+1-p}(X)=\varinjlim_{(U,x)}A_p(U)=\varinjlim_{(U,x)}\varinjlim_{Y}A_p(Y). 
\end{equation}
Note that $A^{d+1-p}(U)$ is generated by $A_p(Y)$ with $Y$ flat over $\O_K$.
In fact, if $y\in Y^{(0)}$ is contained in the special fibre of $X$ over $\O_K$, then by \cite[Lem. 7.2]{SS10} one can find an integral closed subscheme $Z\subset X$ of dimension $p+1$ satisfying the following conditions:
\begin{enumerate}
\item Let $X_K$ be the generic fibre of $X$ over $\O_K$, then $Z\cap X_K\neq \emptyset$.
\item $y\in Z$. 
\item $Z$ is regular at $y$.
\end{enumerate}
Let $\{\bar{\alpha}_1,...,\bar{\alpha}_n\}\in K^M_n(y)$. Let $Z_0=\cup_{i\in I} Z_0^{i}\cup \overline{\{y\}}$ be the union of the pairwise different irreducible components of the special fibre of $Z$ with those irreducible components different from $\overline{\{y\}}$ indexed by $I$. 
Since all maximal ideals, $\mathfrak m_i$ corresponding to $Z_0^{i}$ and $\mathfrak m_y$ corresponding to $\overline{\{y\}}$, in the semi-local ring $\O_{Z,Z_0}$\footnote{$\O_{Z,Z_0}:=\colim_{U\supset S} \O_Z(U)$, where the colimit is taken over all open subsets $U$ of $Z$ containing the set $S$ of generic points of $Z_0$.} are coprime, the map $\O_{Z,Z_0}\r \prod_{i\in I}\O_{Z,Z_0}/\mathfrak m_i\times \O_{Z,Z_0}/\mathfrak m_y$ is surjective. 
Therefore we can lift each $\bar{\alpha}_i$ to some $\alpha_i\in K^M_1(\O_{Z,Z_0})\subset K^M_1(z)$, $z$ being the generic point of $Z$, which specializes to $\bar{\alpha}_i$ in $K(\overline{\{y\}})^{\times}$ and to $1$ in $K(Z_0^{i})^{\times}$ for all $i\in I$. Let $\tilde{\pi}$ be a local parameter at $y$. Then 
$$\partial: K^M_{n+1}(z)\r \oplus_{x\in Z_0^{(0)}}K^M_{n}(x)$$
maps $\{\alpha_1,...,\alpha_n,\tilde{\pi}\}$ to $(0,...,0,\{\bar{\alpha}_1,...,\bar{\alpha}_n\})$.

By Lemma \ref{mainlemma} we have that $A_p(Y)\r A_p(U)$ is the zero map for $Y$ flat over $\O_K$ and $U$ small enough. It follows from (\ref{equation_proof_exactness_right}) and the last paragraph that $A^{p}(X)=0$ for $p>0$.
\end{proof}

We summarise the proof as follows: we showed that elements in the Gersten complex being supported on relative or horizontal, i.e. flat, subschemes of codimension $>0$ vanish after taking homology groups. But since we are only concerned with the part of the Gersten complex where the codimension is $>0$, all the supports can be moved into horizontal position.

\section{Exactness at the first and second place}\label{section_5}
%We fix the following notation: Let $\mathcal{O}_K$ be a $p$-adic local ring with quotient field $K$ and residue field $k=\mathcal{O}_K/\pi\mathcal{O}_K$.
In this section we prove our main theorem.
\begin{theorem}\label{thm1intext}
Let $\mathcal{O}_K$ be an excellent discrete valuation ring with infinite residue field $k$ and local parameter $\pi$. Let $X$ be a smooth excellent irreducible scheme over $\Spec(\mathcal{O}_K)$. Let $A$ be the semi-local ring associated to a set of closed points $x_1,...,x_l$ 
%of $X$ which are all contained in the same component 
of the special fibre of $X$. Assume that the map
$$K^M_n(A_{(\pi)})\xrightarrow{} K^M_n(\mathrm{Frac}(A))$$
is injective. Then the map
$$K^M_n(A)\xrightarrow{} K^M_n(\mathrm{Frac}(A))$$
is injective.
\end{theorem}
\begin{proof}
%By a standard norm argument we may assume that $k$ is infinite. 
Let $d$ be the relative dimension of $X$ over $\O_K$. 
Consider the sequence of morphisms
\begin{equation}\label{inductionassumtion}
A\r \varinjlim_{}A_f=A_{(\pi)}\r A_{(\pi)}[1/\pi]=\mathrm{Frac}(A),
\end{equation}
where the colimit is taken over all $f\in A$ such that $V(f)$ is flat over $\Spec(\mathcal{O}_K)$. We are going to show that the first morphism induces an injection on Milnor K-groups.
We proceed by induction on $d$. Our induction assumption is the following: the first arrow in sequence (\ref{inductionassumtion}) induces an injection on Milnor K-groups for any $A$ as in the statement of the theorem if the relative dimension of $\Spec(A)$ over $\O_K$ is $\leq d-1$. The case $d=0$ is trivial.

Let $a\in \mathrm{ker}[K^M_n(A)\xrightarrow{} K^M_n(A_{(\pi)})]$. Then 
%for $\tilde{A}$ the local ring of $x$ in some Nisnevich neighborhood $\tilde{X}$ of $X$ at $x$. Furthermore 
there is some $f\in A$ such that $a\in\mathrm{ker}[K^M_n(A)\r K^M_n(A_f)]$. By Proposition \ref{geomrep2} and possibly shrinking $X$ around $x_1,...,x_l$, we may assume that there is a morphism
$$p_{\underline{u}}: X\r \mathbb{A}^d_{\O_K}$$
with the following properties:
\begin{enumerate}
\item The map $V(f)\r \mathbb{A}^d_{\O_K}$ is an embedding.
\item $p_{\underline{u}}$ is \'etale.
\item Let $y_1,...,y_l$ denote the images of the closed points of $A$ under $p_{\underline{u}}$. Let $A'$ be the semi-local ring at $y_1,...,y_l$ and let $f'\in A'$ be chosen according to (1) such that 
$A/f\cong A'/f'$. Then $A/f\cong A\otimes_{A'}A'/f'$. 
\end{enumerate}
Properties (1) and (2) follow from Proposition \ref{geomrep2}(3) and (3) follows from Proposition \ref{geomrep2}(4).
By Theorem \ref{Nisdescent} there is a co-Cartesian square
\[
  \xymatrix{
    K^M_n(A')\ar[r] \ar[d] & K^M_n(A'_{f'})
    \ar[d]
  \\
  K^M_n(A)  \ar[r] & K^M_n(A_f).
  }
\]
Since the square is co-Cartesian, the injectivity of the upper horizontal map would imply the injectivity of the lower horizontal map. It therefore suffices to show that the map
$$ K^M_n(A')\r  K^M_n(A'_{f'})$$
is injective. Note that $\pi$ does not divide $f'$. Let $a$ be in the kernel of this map and denote by $p_1,...,p_m\in \O_K[t_1,...,t_d]$ the irreducible polynomials appearing in the symbols of $a$ with $p_i\in A'^*$. Let $q_1,...,q_{m'}\in \O_K[t_1,...,t_d]$ be the irreducible polynomials appearing in the relations which make $a$ zero in $ K^M_n(A'_{f'})$. Let $W_i\subset V(p_i), i=0,...,m$ be the join of the singular locus of $V(p_i)$ with $\bigcup_{j\neq i}V(p_i)\cap V(p_j)$.

Next we construct a linear projection
$$p:\mathbb{A}^d_{\O_K}\r \mathbb{A}^{d-1}_{\O_K}$$
such that $p|_{V(p_j)}$ and $p|_{V(q_j)}$ are finite and such that $p(y_i)\notin p(W_j)$ for all $i,j$.
%\footnote{In the notation of Lemma \ref{geomrep1}: $\mathbb{A}^d_{\O_K}=X_0=X$, $p=\pi$, $S=\{y_1,...,y_l\}$ and $Y=\cup_j V(p_j)\cup_j V(q_j)$. Note that we do not need to localise $\mathbb{A}^d_{\O_K}$. We can start with a projection on the special fiber $p:\mathbb{A}^d_k\to \mathbb{A}^{d-1}_k$ satisfying the above properties. Such a projection is determined by the images of  and then lift it as in the proof of \cite[Lem. 1]{GL87}.} 
First, by the following argument  we may assume that the special fibres of ${V(p_j)}$ and ${V(q_j)}$ are dense in the special fibres of their closures $\overline{V(p_j)}$ and $\overline{V(q_j)}$: replace $V(p_i)$ (and $V(q_i)$) by $\overline{V(p_i)}$ in $\mathbb{P}^d_{\O_K}$ and take linear form $l$ such that the hyperplane $V(l)$ does not contain any of the connected components of the special fibre of the $\overline{V(p_i)}$ or any of the $y_i$. In this case $\overline{V(p_i)}\cap (\mathbb{P}^d_{\O_K}\setminus V(l))$ is dense in $\overline{V(p_i)}$.
Then by Proposition \ref{proposition_hyperplanes} one can choose a projection 
$$p:\mathbb{A}^d_{\O_K}\cong\mathbb{P}^d_{\O_K}\setminus V(l)\r \mathbb{A}^{d-1}_{\O_K}$$
which satisfies the required properties. The finiteness property follows from \cite[Lemma 2.3]{SS18}.

We now continue to follow the proof in \cite[Sec. 6]{Ke09}. Let $A''$ be the semi-local ring associated to the points $p(y_1),...,p(y_l)\in \mathbb{A}^{d-1}_{\O_K}$. Then $A''\subset A'$ is a local ring extension and because $p|_{V(p_i)}$ is finite it follows that $p_i\in A''[t]$ is irreducible and can be chosen to be monic. Let $q\in A''[t] $ be a monic polynomial such that the intersection of $V(q)$ with $p^{-1}(p(y_i))$ consists exactly of the points $y_1,...,y_l$ which are in the fibre of $p(y_i)$ for all $i=1,...,l$. Then $(q,p_i)=1$ for all $i=0,...,m$. By our choice of $q$, there is a natural map 
$$K^t_n(A'',q)\r K^M_n(A')$$
sending a feasible $n$-tuple in the sense of Definition \ref{definition_feasible_tuples}(5) to the same tuple, or rather the induced symbol, in $ K^M_n(A')$. That $(q,p_i)=1$ implies that $a$ is in the image of some element $a'\in K^t_n(A'',q)$. We show that $a'=0$. Consider the commutative diagram 
%$$0\r K^M_n(A)\r K^t_n(A,p)\r \oplus_{\pi}K^M_{n-1}(A[t]/(\pi))\r 0$$

\[
  \xymatrix{
   0\ar[r] & K^M_n(A'')\ar[r] \ar[d] & K^t_n(A'',q)\ar[d] \ar[r]^-\partial & \bigoplus_{g}K^M_{n-1}(A''[t]/(g)) \ar[d] \ar[r] & 0
  \\
 0\ar[r] &  K^M_n(A''_{(\pi)})  \ar[r] & K^t_n(A''_{(\pi)})\ar[r] & \bigoplus_{g}K^M_{n-1}(A''_{(\pi)}[t]/(g))  \ar[r] & 0
  }
\]
whose rows are exact by Proposition \ref{milnorbasstate}. The direct sum in the upper row is taken over all monic irreducible $g\in A[t]$ with $(g,q)=1$. The direct sum in the lower row is taken over all monic irreducible $g\in A''_{(\pi)}[t]$. First note that because $p|_{V(q_i)}$ is finite it follows that $a'$ maps to zero in $K^t_n(A''_{(\pi)})$. The left vertical map is injective by our induction assumption and $A''[t]/(g)$ is regular for all $g$ where $\partial(a')$ is possibly $\neq 0$. On these components the right vertical map is therefore also injective by our induction assumption. This implies that $a'=0$. 
\end{proof}

\begin{theorem}\label{thm2intext}
Let $\mathcal{O}_K$ be an excellent discrete valuation ring with infinite residue field $k$. Let $X$ be a smooth excellent irreducible scheme over $\Spec(\mathcal{O}_K)$. Let $x\in X$ be a closed point in the special fibre of $X$ over $\mathcal{O}_K$ and $A:=\O_{X,x}$. 
Then the sequence
\begin{equation}\label{eqn2}
K^M_n(A)\xrightarrow{} K^M_n(\mathrm{Frac}(A))\r \bigoplus_{x\in \Spec(A)^{(1)}} K^M_{n-1}(x)
\end{equation}
is exact. 
%In other words the exactness of sequence (\ref{eqn2}) can be reduced to the exactness of sequence (\ref{eqn1}).
\end{theorem}

\begin{proof}
Let $\eta$ be the generic point of the component of the special fibre of $X$ which contains $x$ and let $A_\eta:=\O_{X,\eta}$. The sequence
\begin{equation}\label{eqn1}
K^M_n(A_\eta)\xrightarrow{} K^M_n(\mathrm{Frac}(A))\r K^M_{n-1}(\eta)
\end{equation} 
is exact (even on the right) by \cite[Prop. 7.1.7]{GS06}; in \textit{loc. cit.} $U_n$ is exactly the image of $K^M_n(A)$ in $K^M_n(K)$.
Our contribution is to show that the sequence
\begin{equation}\label{eqn3}
K^M_n({A})\xrightarrow{} K^M_n(A_\eta)\r \bigoplus_{x\in \Spec(A)^{(1)}-\eta} K^M_{n-1}(x)
\end{equation}
is exact. Before proving this, note that the exactness of the sequences (\ref{eqn1}) and (\ref{eqn3}) imply the exactness of sequence (\ref{eqn2})\footnote{We would like to emphasise the analogy with Theorem \ref{thm1intext}: in both theorems the general case is reduced to the case of a discrete valuation ring.}: consider the diagram
\[
  \xymatrix{
    &  K^M_n(\mathrm{Frac}(A))  \ar[r] &  K^M_{n-1}(\eta)
  \\
  K^M_n(A)  \ar[r] & K^M_n(A_\eta) \ar[r] \ar[u] &  \underset{x\in \Spec(A)^{(1)}-\eta}{\bigoplus} K^M_{n-1}(x).
  }
\]
If $a\in \mathrm{ker}[K^M_n(\mathrm{Frac}(A))\r \bigoplus_{x\in \Spec(A)^{(1)}} K^M_{n-1}(x)]$, then $a\in \mathrm{ker}[K^M_n(\mathrm{Frac}(A))\r K^M_{n-1}(\eta)]$ and therefore by (\ref{eqn1}), $a\in \mathrm{ker}[K^M_n(A_\eta)\r \bigoplus_{x\in \Spec(A)^{(1)}-\eta} K^M_{n-1}(x)]$ which by the exactness of (\ref{eqn3}) implies that $a\in K^M_n(A)$.

We now return to proving the exactness of sequence (\ref{eqn3}). We proceed by induction on the relative dimension $d$ of $X$ over $\O_K$ in (\ref{eqn3}). Assume that (\ref{eqn3}) is exact if the relative dimension is $\leq d-1$. Note that (\ref{eqn3}) is exact by definition if $d=0$.

From here on we closely follow the proof given in \cite[Sec. 4.2 and 4.3]{Ke05}. First note that every element $a\in\mathrm{ker}[K^M_n(A_\eta)\r \bigoplus_{x\in \Spec(A)^{(1)}-\eta} K^M_{n-1}(x)]$ is induced by an element $a_0\in K^M_n(A[1/f])$ for some $f\in A\setminus \{0, \pi A\}$ and $\partial_x(a_0)=0$ for every $x\in (V(f))^{(0)}$. By Proposition \ref{geomrep2} and possibly shrinking $X$ we can find an \'etale morphism 
$$p_{\underline{u}}:X\r \mathbb{A}^d_{\O_K}$$
such that there is an $f'\in \O_{\mathbb{A}^d_{\O_K},p_{\underline{u}}(x)}=:A'$ and an isomorphism
$$p_{\underline{u}}^*:A'/(f')\r A/(f).$$
Consider the commutative diagram
\begin{equation}\label{diagram_reduction_Ad}
  \xymatrix{
  K^M_n(A') \ar[d]\ar[r]  &  K^M_n(A'[1/f'])  \ar[r] \ar[d]^j & \underset{x\in V(f')^{(0)}}{\bigoplus} K^M_{n-1}(x)  \ar[d]^{\cong}
  \\
  K^M_n(A)  \ar[r]^i & K^M_n(A[1/f]) \ar[r]  &  \underset{x\in V(f)^{(0)}}{\bigoplus} K^M_{n-1}(x).
  }
\end{equation}
in which the morphism on the right is an isomorphism since $p_{\underline{u}}^*$ is an isomorphism. By Lemma 4.2.2 of \textit{loc. cit.} we have that $K^M_n(A[1/f])=\mathrm{im}(i)+\mathrm{im}(j)$. Note that the proof of \textit{loc. cit.} is formulated in the context of smooth varieties over a field but transfers word by word to our situation. By a diagram chase this reduces the exactness of the second row to showing the exactness of the first row.

Before we proceed we introduce some notation. Let $Y$ be a smooth connected excellent scheme over $\Spec(\mathcal{O}_K)$. Let $\tilde{\eta}$ be the generic point of $Y$. Let $\{\eta_i\}$ be the set of generic points of the special fibre of $Y$. Set $$\bar{K}^M_n(Y):=\mathrm{ker}[K^M_n(\tilde \eta)\r \bigoplus_{x\in Y^{(1)}-\{\eta_i\}} K^M_{n-1}(x)].$$

We now show that the Theorem \ref{thm2intext} holds for $\mathbb{A}^d_{\O_K}$. This then in particular implies that the top row of diagram $(\ref{diagram_reduction_Ad})$ is exact and therefore by the above that Theorem \ref{thm2intext} holds in general. We still proceed by the same induction on the relative dimension $d$.
For $y\in (\mathbb{A}^d_{\O_K})^{(d+1)}$ we have that
$$\underset{y\in U\subset \mathbb{A}^d_{\O_K}}{\varinjlim} \bar{K}^M_n(U)=\mathrm{ker}[K^M_n(\tilde \eta)\r \bigoplus_{x\in \Spec(\O_{\mathbb{A}^d_{\O_K},y})^{(1)}-\eta} K^M_{n-1}(x)],$$
where $\eta$ is the generic point of the special fibre of $\mathbb{A}^d_{\O_K}$ and the colimit is taken over all open subschemes of $\mathbb{A}^d_{\O_K}$ containing $x$.
In particular, every element $a\in \mathrm{ker}[K^M_n(\tilde \eta)\r \bigoplus_{x\in \Spec(\O_{\mathbb{A}^d_{\O_K},y})^{(1)}-\eta} K^M_{n-1}(x)]$ extends to an element $\bar{a}\in \bar{K}^M_n(U)$ for some $U\subset \mathbb{A}^d_{\O_K}$ containing $y$. Let $D_i, i=1,...,\lambda$, be the irreducible divisors of $\mathbb{A}^d_{\O_K}\setminus U $. The $D_i$ are flat over $\mathcal{O}_K$ since $y$ is in the special fibre, which has just one component, and $U$ is open. As in the proof of Theorem \ref{thm1intext}, we choose $U$ and $\mathbb{A}^d_{\O_K}$ such that for each $i$ the special fibre of $D_i$ is dense in the special fibre of the closure of $D_i$ in $\mathbb{P}^d_{\O_K}$. For all $i$ let $\xi_i$ denote the generic point of $D_i$. Note that $\partial_{\xi_i}(\bar{a})\in K^M_{n-1}(\xi_i)$ is not necessarily zero but in the image of $K^M_{n-1}(H^0(\Omega_i,\O_{\Omega_i}))$\footnote{Here and in the following the Milnor $K$-theory $K^M_n(R)$ of a (non-local) ring $R$ is defined naively as the $n$-fold tensor product $R\otimes...\otimes R$ modulo relations generated by $\{a\otimes (1-a)|a,1-a\in R^\times\}$ and $\{a\otimes (-a)|a\in R^\times\}$.} for some affine open dense subscheme $\Omega_i\subset D_i$. We choose a projection
$$p:\mathbb{A}^d_{\O_K}\r \mathbb{A}^{d-1}_{\O_K}$$
such that for all $i$
\begin{enumerate}
\item $p$ is finite when restricted to $D_i$ (again by \cite[Lemma 2.3]{SS18}), and
\item $p(y)\notin p(D_i\setminus\Omega_i)$.
\end{enumerate}
Let $A''$ be the local ring $\O_{\mathbb{A}^{d-1}_{\O_K},p(y)}$ and consider the base change 
$$p:\mathbb{A}^1_{A''}\r \Spec(A'').$$
Let $q\in A''[t] $ be a monic such that the intersection of $V(q)$ with $p^{-1}(p(y_i))$ consists exactly of $y$. Consider the exact sequence of Theorem \ref{milnorbasstate}
$$0\r K^M_n(A'')\r K^t_n(A'',q)\xrightarrow{\partial} \bigoplus_{g}K^M_{n-1}(A''[t]/(g))\r 0,$$
where the direct sum is taken over all monic irreducible $g\in A''[t]$ with $(g,q)=1$ and $\partial=\oplus \partial_g$. By conditions $(1)$ and $(2)$, one can pick a $b\in K^t_n(A'',q)$ such that $ \partial_g(b)=\partial_{\xi_i}(\bar{a})=\partial_{g}(\bar{a})$ and $0$ otherwise. This implies that $\bar{a}-b\in \bar{K}^M_n(\mathbb{A}^1_{A''})$; indeed, by the definition of feasible in \ref{definition_feasible_tuples}, $\partial_x(b)=0$ for all $x$ of codimension one which are not given by monic irreducible polynomials. By our choice of $q$, and since $y$ is not contained in any of the $D_i$, we get that $b$ induces an element in $K^M_n(\O_{\mathbb{A}^d_{\O_K},y})$. Therefore, if $\bar{a}-b\in \bar{K}^M_n(\mathbb{A}^1_{A''})$ is in the image of $K^M_n(A''[t])$, then also $a$ is in the image of $K^M_n(\O_{\mathbb{A}^d_{\O_K},y})$. We may therefore assume that $\bar{a}\in \bar{K}^M_n(\mathbb{A}^1_{A''})$. But $\bar{K}^M_n(\Spec(A''))\cong \bar{K}^M_n(\mathbb{A}^1_{A''})$ by homotopy invariance (see below), and by our induction assumption the map $K^M_n(A'')\r \bar{K}^M_n(\Spec(A''))$ is surjective. This is summarised in the commutative diagram 
\[
  \xymatrix{
   K^M_n(A''[t]) \ar[r]^-{}  &  \bar{K}^M_n(\Spec(A''[t]))    & 
  \\
 K^M_n(A'') \ar[u] \ar@{->>}[r] & \ar[u]_{\cong} \bar{K}^M_n(\Spec(A'')).
  }
\]
For the homotopy invariance of $\bar{K}^M_n$ note that we have a commutative diagram with exact rows and columns
\[
  \xymatrix{
   & 0 \ar[d] & 0 \ar[d] & & \\
 & \bar{K}^M_n(A'') \ar[d]\ar[r]^-{h}  &  \bar{K}^M_n(\Spec(A''[t]))   \ar[d] & &
  \\
 0 \ar[r] & K^M_n(\mathrm{Frac}(A'')) \ar[d] \ar[r] & K^M_n(\mathrm{Frac} (A'')(t)) \ar[r]^-{\partial} \ar[d]^{\partial'} &  \underset{\pi}{\bigoplus} K^M_{n-1}(\mathrm{Frac} (A'')[t]/\pi)\ar[r]& 0 \\
 & \underset{x\in \Spec(A'')^{(1)}-\eta''}{\bigoplus} K^M_{n-1}(x) \ar[r] & \underset{x\in \Spec(A''[t])^{(1)}-\eta''_t}{\bigoplus} K^M_{n-1}(x) & &
  }
\]
where in the bottom row $\eta''$ denotes the generic point of the special fibre of $\Spec(A'')$ and $\eta''_t$ the generic point of the special fibre of $\Spec(A''[t])$. The columns are exact by the definition of $\bar{K}^M_n$. The middle row is the original Milnor-Bass-Tate sequence for fields (Theorem \ref{milnorbasstate_fields}), in which the sum is taken over all monic irreducible polynomials $\pi\in \mathrm{Frac} (A'')[t]$. The injectivity of $h$ follows from the commutativity of the upper square and the surjectivity from the fact that $\mathrm{ker}(\partial')\subset \mathrm{ker}(\partial)$ and that the lower horizontal map is injective. Indeed, if $F$ is a field, then the map $K_{n-1}^M(F)\to K_{n-1}^M(F(t))$ is injective since it has an inverse given by specialisation at infinity, which associates to a polynomial its highest coefficient.
\end{proof}

\begin{corollary}\label{maincorollary}
Let $\mathcal{O}_K$ be an excellent discrete valuation ring with residue field $k$. Let $X$ be a smooth scheme over $\Spec(\mathcal{O}_K)$. 
Then the sequence of sheaves
$$\hat{\mathcal{K}}^M_{n,X}\xrightarrow{i} \bigoplus_{x\in X^{(0)}}i_{x*}K^M_n(x)\rightarrow \bigoplus_{x\in X^{(1)}}i_{x*}K^M_{n-1}(x)\r... \bigoplus_{x\in X^{(d)}}i_{x*}K^M_{n-d}(x)\r 0$$
is exact, in which $d=\dim X$. Furthermore, if $\O_{X,\eta}$ is the discrete valuation ring induced by a generic point of the special fibre, then $i$ is injective if the map
$$\hat{K}^M_{n}(\O_{X,\eta})\r K^M_{n}(\mathrm{Frac}(\O_{X,\eta}))$$
is injective for all such $\eta$.
\end{corollary}
\begin{proof}
%By Neron--Popescu desingularisation we may assume that $k$ is perfect: we only need to take care of the mixed characteristic case. Let $\mathbb{F}_p$ be the prime field of $k$. Then the morphism $X\to \Spec(\mathbb{Z}_{(p)})$ is regular\todo{no. take maximal ramified extension} and therefore by Neron--Popescu $X$ is a limit of smooth schemes over $\Spec(\mathbb{Z}_{(p)})$. 
We can check the statements locally on $X$.
Let $y\in X$ be a (not necessarily closed) point. We may assume that $y$ is contained in the special fibre of $X$ over $\Spec(\mathcal{O}_K)$ since the generic fibre is open in $X$ and the corollary is known in equal characteristic by \cite[Prop. 10(8)]{Ke10} (see also Remark \ref{remark_introduction}(1)). Let $\eta$ be the generic point of the component of the special fibre of $X$ which contains $y$.
We need to show that the sequence
\begin{equation*}
 \hat K^M_n({\O_{X,y}})\xrightarrow{\alpha} K^M_n(\mathrm{Frac}(\O_{X,y}))\xrightarrow{\beta} \bigoplus_{x\in \Spec(\O_{X,y})^{(1)}} K^M_{n-1}(x)\r ...\r \bigoplus_{x\in \Spec(\O_{X,y})^{(d)}}K^M_{n-d}(x)\r 0
\end{equation*}
is exact and that $\alpha$ is injective if we assume that $\hat K^M_{n}(\O_{X,\eta})\r K^M_{n}(\mathrm{Frac}(\O_{X,\eta}))$ is injective.
By Theorem \ref{thmrightexactintext} it suffices to show that the sequence
\begin{equation}\label{eqn_4.6}
\hat K^M_n({\O_{X,y}})\xrightarrow{\alpha} K^M_n(\mathrm{Frac}(\O_{X,x}))\xrightarrow{\beta} \bigoplus_{x\in \Spec(\O_{X,y})^{(1)}} K^M_{n-1}(x)
\end{equation}
is exact at the second place and that $\alpha$ is injective if $\hat K^M_{n}(\O_{X,\eta})\r K^M_{n}(\mathrm{Frac}(\O_{X,\eta}))$ is injective. 
If $y$ is closed and the residue field of $\O_{K}$ is infinite, then (\ref{eqn_4.6}) is exact by Theorem \ref{thm2intext}, and $\alpha$ is injective under the given assumption by Theorem \ref{thm1intext}. 

If $y$ is not closed, but still assuming that $k$ is infinite, we can reduce to the case of $y$ being closed as follows: 
for each element $a$ in the kernel of $\alpha$ we can choose a closed point $y'$ contained in $\overline{\{y\}}$ such that $a$ is induced by an element of $\hat K^M_n({\O_{X,y'}})$ (in fact, since the residue field of $\O_K$ is infinite, $\hat K^M_n({\O_{X,y'}})\cong K^M_n({\O_{X,y'}})$). The statement then follows from the commutative diagram
\[
  \xymatrix{
   \hat K^M_n(\O_{X,y}) \ar[r]^-{}  &  \hat{K}^M_n(\O_{X,\eta})    \ar@{^{(}->}[r] & {K}^M_n(\mathrm{Frac}(\O_{X,y}))
  \\
 \hat K^M_n(\O_{X,y'}) \ar[u] \ar@{^{(}->}[r] & \ar@{=}[u]_{} \hat{K}^M_n(\O_{X,\eta}) \ar@{^{(}->}[r] & \ar@{=}[u]_{} {K}^M_n(\mathrm{Frac}(\O_{X,y'}))
  }
\]
since by the last paragraph we know that the horizontal arrow on the lower left is injective. If $b$ is in the kernel of $\beta$, then let $\{x_i\}$ be the finite set of codimension one points of $X$, where $\partial_x(b)\neq 0$, and $D=\cup \overline{\{x_i\}}$ be the corresponding divisor. Then, by assumption, $y\notin D$, and therefore there exists a closed point $y'\in \overline{\{y\}}-\overline{\{y\}}\cap D$.
The statement then follows from the commutative diagram
\[
  \xymatrix{
  \hat K^M_n(\O_{X,y}) \ar[r]^-{}   &  {K}^M_n(\mathrm{Frac}(\O_{X,y}))    \ar[r] &  \bigoplus_{x\in \Spec(\O_{X,y})^{(1)}} K^M_{n-1}(x)
  \\
 \hat K^M_n(\O_{X,y'}) \ar[u]  \ar[r] & \ar@{=}[u]_{} {K}^M_n(\mathrm{Frac}(\O_{X,y'})) \ar[r] &   \bigoplus_{x\in \Spec(\O_{X,y'})^{(1)}} K^M_{n-1}(x)
  }
\]
since by the last paragraph we know that the lower sequence is exact and $b$ is also in the kernel of the lower right map.

Finally, it follows from a norm argument, along the lines of the proof of the isomorphism $G(A)\to \hat{G}(A)$ for $G\in \mathfrak{CT}$ in the proof of \cite[Prop. 9]{Ke10}, that we can drop the assumption that the residue field of $\O_{K}$ is infinite. 
We give the argument for this reduction for the injectivity, the exactness at the second place can be shown similarly.
We only need to worry about the case that $\O_K$ has finite residue field $k$. Let $R:=\O_{X,y}$ and let $L$ denote its residue field, which is a finitely generated, separable field extension of $k$. So we may realise $L$ as a finite separable extension of a rational function field $k(\ul t):=k(t_1,\dots,t_d)$.
We fix a prime $p$ which is coprime to $|K:k(\ul t)|$, and let $k^{(r)}$ be the unique degree $p^r$ extension of the finite field $k$. Note that $L\otimes_kk^{(r)}=L\otimes_{k(\ul t)} k^{(r)}(\ul t)$ is the tensor product of finite field extensions of coprime degree, hence is a field.
Let $\O_K^{(r)}$ be the finite \'etale extension of $\O_K$ corresponding to the extension $k^{(r)}$ of $k$, and set $R^{(r)}:=R\otimes_{\O_K}\O_K^{(r)}$. Note that $R^{(r)}$ is local since $R^{(r)}/\mathfrak m_RR^{(r)}=L\otimes_kk^{(r)}$ is a field, and essentially smooth over $\O_K^{(r)}$. By the same argument, we may additionally assume that, denoting $R_\eta:= \O_{X,\eta}$, the ring $R_\eta^{(r)}:= R_\eta\otimes_{\O_K}\O_K^{(r)}$ is local and essentially smooth.  
Consider the commutative diagram
\[
  \xymatrix{
  \hat K^M_n(R^{(r)}) \ar[r]^-{} \ar@{-->}@/^5mm/[d]^N  &  {K}^M_n(R_\eta^{(r)})   \ar@{-->}@/^5mm/[d]^N &
  \\
 \hat K^M_n(R) \ar[u]  \ar[r] & \ar[u]_{} {K}^M_n(R_\eta)) \ar@{^{(}->}[r] & {K}^M_n(\mathrm{Frac}(R)), 
  }
\]
in which the dotted arrows correspond to the norm map $N$ defined in \cite{Ke09,Ke10}. Let $a$ be in the kernel of the lower horizontal map. For primes $p$ and $q$ as above with $(p,q)=1$, by a filtered colimit argument we get that for $r$ large enough $p^ra=0=q^ra$ and therefore $a=0$.

%In order to reduce to $k$ being infinite we consider for a fixed prime $p$ a tower of finite \'etale extensions 
%$$A\subset A_1 \subset A_2 \subset...\subset A_\infty$$
%such that $A_m$ is local, $[A_{m+1}:A_m]=p$ and $A_\infty=\cup_m A_m$. Considering two such towers of prime degrees coprime to each other it follows from the continuity of Milnor K-theory (see \cite[Prop. 6]{Ke10}) and the fact that the composition $\hat K^M_n(A_m)\to \hat K^M_n(A_{m+1})\to \hat K^M_n(A_m)$ is multiplication by $p$ that we may assume that $k$ is infinite.
\end{proof}

\begin{remark}[Finite coefficients]
Corollary \ref{maincorollary} also holds with finite coefficients, i.e. if we replace $\hat K^M_n$ by $\hat K^M_n/m$ and $K^M_n$ by $K^M_n/m$. Indeed, the proofs go through exactly the same way. Note for example that the Milnor-Bass-Tate sequence of Theorem \ref{milnorbasstate} remains exact after tensoring with $\otimes_{\mathbb Z}\mathbb Z/n\mathbb Z$ since it is split exact.
\end{remark}

%\section{SNC}

%\begin{proposition}
%Let $Y$ be  a simple normal crossing variety. Then the natural map
%$$\mathcal{K}^M_{Y,n}\r C^\bullet(Y,n)$$
%is a quasi-isomorphism.
%\end{proposition}
%\begin{proof}

%\end{proof}

\section{Gersten conjecture for henselian DVRs with finite residue field}\label{Gerstenreldimzero}

In this section we prove the Gersten conjecture for henselian DVRs with finite residue field. The case of a complete DVR with finite residue field is proved by Dahlhausen in \cite[Cor. 4.4]{Da18}. We deduce our result from his using an algebraization technique. The result is independent of the other results of the article.
\begin{theorem}\label{reldimzero}
Let $\O$ be an excellent henselian discrete valuation ring with quotient field $F$ and finite residue field $\kappa$. Then the sequence
$$0\r \hat{K}^M_n(\O)\r {K}^M_n(F) \r K^M_{n-1}(\kappa)\r 0$$
is exact.
\end{theorem}
\begin{proof}
The exactness of $\hat{K}^M_n(\O)\r {K}^M_n(F) \r K^M_{n-1}(\kappa)\r 0$ is shown in \cite[Prop. 2.7]{Da18}. It therefore remains to be shown that the map $\hat{K}^M_n(\O)\r \hat{K}^M_n(F)$ is injective. Let $\hat{\O}$ be the completion of $\O$ along the maximal ideal $\mathfrak m$ of $\O$. By \cite[7.8.3(v)]{EGA4II} and since $\O$ is excellent the map $\O\r \hat{\O}$ is geometrically regular. Therefore by N\'eron-Popescu desingularization $\O\r \hat{\O}$ is a filtered colimit of smooth $\O$-algebras $A_i$, i.e. $\varinjlim A_i=\hat{\O}$. By Elkik's theorem (see Theorem \ref{elkikthm} below) the map $\O\r \O/\mathfrak m$ has the right lifting property with respect to smooth maps. In particular for all $i$ the dotted arrow in the diagram 
\[
  \xymatrix{
    \O \ar[rr] \ar[d] &  & \O
    \ar[d]
  \\
  A_i  \ar[r] \ar@{.>}[rru] & \hat{\O} \ar[r] & \O/\mathfrak m
  }
\]
exists giving a split injection $\O\r A_i$. Improved Milnor $K$-theory is a covariant functor from rings to abelian groups and therefore the induced map $\hat{K}^M_n(\O)\r \hat{K}^M_n(A_i)$ is also injective. Since filtered colimits of abelian groups are exact and improved Milnor $K$-theory commutes with filtered colimits (see \cite[Prop. 6 \& Thm. 7]{Ke10}), this implies that the map $\hat{K}^M_n(\O)\r \varinjlim \hat{K}^M_n(A_i)\cong \hat{K}^M_n(\hat{\O})$ is injective. The statement now follows from the commutative diagram

\[
  \xymatrix{
    \hat{K}^M_n(\hat{\O}) \ar@{^{(}->}[r]  & K^M_n(\mathrm{Frac}(\hat{\O}))
  \\
 \hat{K}^M_n(\O)  \ar[r]  \ar@{^{(}->}[u] & K^M_n(F) \ar[u]
  }
\]
and the fact that the upper horizontal map is injective by \cite[Cor. 4.4]{Da18}.
\end{proof}

In the above proof we needed the following lifting property of henselian pairs
with respect to smooth morphisms. 

\begin{theorem}[{Elkik \cite[Sec.~II]{Elkik1973}}] 
\label{elkikthm}
Let $(S, I)$ be a henselian pair. Then $S \to S/I$ has the right lifting
property with respect to smooth maps. That is, any commutative diagram
\begin{equation*} \label{lifthens}  \xymatrix{
A \ar[d] \ar[r] &  S \ar[d]  \\
B \ar[r] \ar@{-->}[ru] &  S/I\
}\end{equation*}
with $A \to B$ smooth (rather than \'etale) admits a lift as in the dotted arrow.
\end{theorem} 

\bibliographystyle{siam}
\bibliography{Bibliografie}
\end{document}